\documentclass[11pt]{article}
\usepackage{amsmath,amssymb,amsthm,amsfonts,epsfig,enumerate}
\usepackage{mathtools}
\usepackage{multicol}
\usepackage{pst-plot,pst-node}
\usepackage{mathrsfs,amssymb}
\usepackage{graphicx}
\usepackage{color}
\usepackage{enumerate}

\usepackage{epstopdf}
\usepackage{url}
\usepackage[unicode]{hyperref}
\definecolor{ForestGreen}{rgb}{0.1,0.6,0.05}
\definecolor{EgyptBlue}{rgb}{0.063,0.1,0.6}
\hypersetup{
	colorlinks=true,
	linkcolor=EgyptBlue,         
	citecolor=ForestGreen,
	urlcolor={EgyptBlue}
}
\usepackage[hyperpageref]{backref}

\usepackage{amsmath,amssymb,amsthm,amsfonts,epsfig,enumerate}
\usepackage{mathtools}
\usepackage{multicol}
\usepackage{pst-plot,pst-node}
\usepackage{mathrsfs,amssymb}
\usepackage{graphicx}
\usepackage{color}

\newtheorem{definition}{Definition}[section]
\newtheorem{theorem}[definition]{Theorem}

\newtheorem{remark}[definition]{Remark}
\newtheorem{proposition}[definition]{Proposition}
\newtheorem{lemma}[definition]{Lemma}

\numberwithin{equation}{section}
\DeclarePairedDelimiter\abs{\lvert}{\rvert}%
\DeclarePairedDelimiter\norm{\lVert}{\rVert}%

\makeatletter
\let\oldabs\abs
\def\abs{\@ifstar{\oldabs}{\oldabs*}}
\let\oldnorm\norm
\def\norm{\@ifstar{\oldnorm}{\oldnorm*}}
\newcommand{\al} {\alpha}
\newcommand{\pa} {\partial}

\newcommand{\de} {\delta}
\newcommand{\De} {\Delta}

\newcommand{\Om} {\Omega}
\newcommand{\la} {\lambda}

\newcommand{\Gr} {\nabla}
\newcommand{\no} {\nonumber}
\newcommand{\noi} {\noindent}

\newcommand{\ra} {\rightarrow}

\newcommand{\wide}[1]{\widetilde{#1}}

\def\dx{{\,\rm d}x}
\def\dt{{\rm d}t}

\def\dS{{\rm dS}}

\def \tr {\textcolor{red}}
\def\inpr#1{\left\langle #1\right\rangle}

\def\sb2{{{\mathcal D}^{1,2}_0(B_1^c)}}

\def\C{{\mathcal C}}

\def\M{{\mathcal M}}
\def\S{{\mathcal S}}

\def\R{{\mathbb R}}

\def\F{{\mathcal F}}
\def\({{\Big(}}
\def\){{\Big)}}

\def\ws2{{\F_{\frac{N}{2}}}}
\def\l2{{ L^{1,\;\infty}(\log L)^2}}
\def\M2{\mathcal M_{\log L}}

\def\c1Loc{{\C_{loc}^1}}

\title{On the strict monotonicity of the first eigenvalue\\ of the $p$-Laplacian on annuli}
\author{T. V. Anoop, Vladimir Bobkov\footnote{The author was supported by the project LO1506 of the Czech Ministry of Education, Youth and Sports.}\;\footnote{corresponding author}\,, Sarath Sasi}

\date{}
\begin{document}
 \maketitle
 \textbf{Abstract.}
Let $B_1$ be a ball in $\mathbb{R}^N$ centred at the origin and $B_0$ be a smaller ball compactly contained in $B_1$. 
For $p\in(1, \infty)$, using the shape derivative method, we show that the first eigenvalue of the $p$-Laplacian in annulus $B_1\setminus \overline{B_0}$ strictly decreases as the inner ball moves towards the boundary of the outer ball. 
The analogous results for the limit cases as $p \to 1$ and $p \to \infty$ are also discussed.
Using our main result, further we prove the nonradiality of the eigenfunctions associated with the points on the first nontrivial curve of the Fu\v{c}ik spectrum of the $p$-Laplacian on bounded radial domains.


 
\section{Introduction}

Let $\Om \subset \R^N$ be a bounded domain with $N \geq 2$. We consider the following nonlinear eigenvalue problem:
\begin{equation}\label{evp}
\left.
\begin{aligned}
-\Delta_p u &= \lambda |u|^{p-2} u && {\rm in }\ \Om , \\
u&=0 &&{\rm on }\ \partial\Om,
\end{aligned}
\right\}
\end{equation}
where $\la \in \mathbb{R}$ and $\De_p$ is the $p$-Laplace operator given by $ \De_p u:= { \rm div}(|\Gr u|^{p-2}\Gr u)$, $p > 1$.
A real number $\la$  is called an eigenvalue of \eqref{evp} if there exists $u$ in   $W^{1,p}_0(\Om)\setminus \{0\}$ satisfying 
\begin{align*}
 \int_{\Om} | \Gr u|^{p-2} \, \left<\Gr u, \Gr v\right> \dx = \la \int_{\Om} |u|^{p-2} \, u \,v \dx,
\quad \forall \, v\in W^{1,p}_0(\Om),
\end{align*}
and $u$ is  said to be an eigenfunction associated with $\la$.

It is well known that   \eqref{evp} admits a least positive eigenvalue $\la_1(\Om)$ which has the following variational characterization:
$$
\la_1(\Om)= \inf \left\{ \int_{\Om} |\nabla u|^p\dx:u \in W_0^{1,p}(\Om)\setminus \{0\} \text{ with } \norm{u}_p=1 \right\}.
$$
In this article we consider $\Om$  of the form $B_{R_1}(x)\setminus \overline{B_{R_0}(y)}$  with $\overline{B_{R_0}(y)} \subset B_{R_1}(x)$, where $B_r(z)$ denotes the open ball of radius $r>0$ centred at $z\in \R^N$. 
Since the $p$-Laplacian is invariant under orthogonal transformations, it can be easily seen that
$$
\la_1(B_{R_1}(x)\setminus \overline{B_{R_0}(y)})=\la_1(B_{R_1}(0)\setminus \overline{B_{R_0}(se_1)})
$$
for any $x,y \in \mathbb{R}^N$  such that $|x-y|=s$, where $e_1$ is the first coordinate vector.  Let the annular region $B_{R_1}(0)\setminus \overline{B_{R_0}(se_1)}$ be denoted by $\Om_s$ and  let   $$\la_1(s):=\la_1(\Om_s).$$
We are interested in the behaviour of $\la_1(s)$ with respect to $s$ (in other words, with respect to the distance between centres of the inner and outer balls).
The main objective of this article is to show that $\la_1(s)$ is strictly decreasing on $[0, R_1-R_0)$ for any $p>1$.

Apparently the first result in this direction was obtained by Hersch in \cite{Hersch1963}, where he proved (in the case $N=2$, $p=2$ and even for more general annular domains) that $\la_1(s)$ attains its maximum at $s=0$. In \cite{Ramm1998}, Ramm and Shivakumar conjectured\footnote{Later a  proof  for this  conjecture using an argument attributed to M. Ashbaugh was published in arxiv:math-ph/9911040 by the same authors.} that $\la_1(s)$ is strictly decreasing and they gave numerical results to support this claim.  Later this conjecture and its higher dimensional analogue were proved independently by Harrel et al.\ \cite{HarrelKroger} and Kesavan  \cite{Kesavan}. 
Their proofs  mainly rely on  the following expression for $\lambda_1'(s)$ obtained using the Hadamard perturbation formula (see \cite{garabedian1952,sokolowski}):
\begin{equation}\label{Form12D}
\la_1'(s)= -\int\limits_{x \in \partial B_{R_0}(se_1)}\abs{\frac{\pa u_s}{\pa n}(x)}^2 n_1(x)\, \dS(x),
\end{equation}
where $u_s$ is the positive eigenfunction  associated with  $\lambda_1(s)$ with the normalization  $\norm{u_s}_{2} = 1$, and $n_1$ is the first component of $n = (n_1,\dots,n_N),$  the outward unit normal to $\Om_s$. 
In \cite{HarrelKroger,Kesavan}, the authors used the above formula in conjunction with reflection techniques and the strong comparison principle to show that $\lambda_1'(s)$ is negative on $(0,R_1-R_0)$.
For further reading and related open problems on this topic, we refer the reader to the books \cite{ashbaugh2007,henrot2006}.

For general $p>1,$ it is natural to  anticipate that $\la_1(s)$ is strictly decreasing on $[0, R_1-R_0)$. Indeed, we have the following generalization of formula \eqref{Form12D}: 
\begin{equation}\label{Form1}
\la_1'(s) = -(p-1)\int\limits_{x \in \partial B_{R_0}(se_1)}\abs{\frac{\pa u_s}{\pa n}(x)}^p n_1(x)\, \dS(x).
\end{equation}
The above expression was derived in  \cite{Anisa2} using the  Hadamard perturbation formula (shape derivative formula)  for $\la_1'(s)$ obtained in \cite{Melian2001}. However for $p\ne 2$, one lacks a strong comparison principle that  guarantee the strict monotonicity of $\la_1(s).$ More precisely, the strong comparison principle that is applicable for the  nonlinear nonhomogeneous problems of the following type:
\begin{equation}\label{nonhomo}
\begin{aligned}
-\Delta_p u = \lambda |u|^{p-2}u && {\rm in }\ \Om ,
\quad u=g &&{\rm on }\ \partial\Om.
\end{aligned}
\end{equation}
Thus one can not directly extend the ideas of \cite{Ramm1998,HarrelKroger,Kesavan}  to the nonlinear case and establish the strict monotonicity of $\la_1(s)$ for general $p>1$.
Nevertheless, in \cite{Anisa2}, Chorwadwala and Mahadevan could show that $\la_1'(s) \leq 0 $ for all $s \in [0, R_1-R_0)$  using a \textit{weak} comparison principle  proved in \cite{chorwadwala2015faber} for problems of the form \eqref{nonhomo}. However,  the authors of \cite{Anisa2} could  not  rule out even the possibility of $\la_1(s)$ being a constant, due to the absence of the \textit{strong} comparison principle. 
In this article, we bypass the usage of the strong comparison principle and prove the following result.
\begin{theorem}\label{MainTHM}
 Let $p \in (1, \infty)$ and let $\la_1(s)$ be the first eigenvalue of $-\De_p$ on $\Om_s$. Then  
 $$
 \la_1'(0) = 0  ~\text{ and }~ \la_1'(s)<0, \  \forall \,s\in (0, R_1-R_0).
 $$
 In particular, $\la(s)$ is strictly decreasing on $[0, R_1-R_0).$
\end{theorem}

For our proof, we  derive another formula for $\lambda_1'(s)$ (in terms of the normal derivative of $u_s$ on the outer boundary) in the following form:
\begin{equation}\label{Form2}
\la_1'(s) = (p-1)\int\limits_{x \in \partial B_{R_1}(0)}\abs{\frac{\pa u_s}{\pa n}(x)}^p n_1(x)\, \dS(x).
\end{equation}
We obtained the above expression  by considering the perturbations of $\Om_s$ generated by shifts of the outer ball. On the other hand,  formula \eqref{Form1} was obtained in~\cite{Anisa2} by considering the perturbations  generated by shifts of the inner ball. 
If we assume $\lambda_1'(s) = 0$ for some $s \in (0, R_1-R_0)$, then formulas \eqref{Form1} and \eqref{Form2} help us to
show that the first eigenfunction $u_s$ associated with $\la_1(s)$ is radial (up to a translation) in some annular neighbourhoods of the inner and outer boundaries of $\Om_s$. 
This eventually leads to a contradiction.

\medskip
Next we study the monotonicity property of the corresponding limit problems. To avoid the ambiguity, for each $p>1,$  here we denote the first eigenvalue  $\la_1(s)$  by $\la_1(p,s).$  It is known that $\lim\limits_{p \to \infty} \la_1^{1/p}(p,s)$ and  $\lim\limits_{p \to 1} \la_1(p,s)$ exist, see \cite{juutinen1999,Kawohl-Fridman}. We denote the limit functions as below:
\begin{equation*}\label{LL}
\Lambda_\infty(s) := \lim\limits_{p \to \infty} \la_1^{1/p}(p,s)  
\quad \text{and} \quad 
\Lambda_1(s) := \lim\limits_{p \to 1} \la_1(p,s).
\end{equation*}
Now we state  results analogous to Theorem \ref{MainTHM}.

\begin{theorem}\label{THM2}
	Let $\Lambda_\infty(s)$ and $\Lambda_1(s)$ be defined as before. Then $\Lambda_\infty(s)$ and $\Lambda_1(s)$ are continuous on $[0, R_1-R_0)$ and
	\begin{enumerate}
		\item[\normalfont{(i)}] 
			$\Lambda_\infty(s)$  is strictly decreasing on $[0, R_1-R_0)$;
		\item[\normalfont{(ii)}] 
			$\Lambda_1(s)$ is  decreasing on $[0, R_1-R_0)$. Moreover, there exists $s^* \in [0, R_1-R_0)$ such that $\Lambda_1(0) = \Lambda_1(s^*) > \Lambda_1(s)$ for all $s \in (s^*, R_1-R_0)$.
	\end{enumerate}
\end{theorem}
We use a geometric characterization of $\Lambda_\infty(s)$ given in \cite{juutinen1999} for proving part (i), and for the existence of $s^*$ in part (ii) we use a variational characterization of $\Lambda_1(s)$ given in \cite{Kawohl-Fridman}.

\medskip
Finally,  we study the following Fu\v{c}ik eigenvalue problem: 
\begin{equation}\label{F}
\left.
\begin{aligned}
-\Delta_p u &= \alpha (u^+)^{p-1} - \beta (u^-)^{p-1} &&\text{in }~ \Omega,\\
u &= 0 &&\text{on }~ \partial\Omega,
\end{aligned}
\right \}
\end{equation}
where $\alpha,\beta$ are real numbers (spectral parameters) and $u^\pm := \max\{\pm u, 0\}.$ 
If problem \eqref{F} possesses a nontrivial solution for some $(\alpha, \beta)$, then we say that $(\alpha, \beta)$ belongs to the Fu\v{c}ik spectrum of \eqref{F}.

In \cite{Cuesta-Fucik}, the authors considered a set of  critical values $c(t)$ given by
\begin{equation}\label{eq:Fucik}
c(t):= \inf_{\gamma \in \Gamma}
\max_{u \in \gamma[-1,1]} \left(\int\limits_{\Omega}|\nabla u|^p \, \dx - t \int\limits_{\Omega}(u^+)^p \, \dx
\right),
\end{equation}
where
\begin{align}
\label{Gamma}
\Gamma &:= \{\gamma \in \C([-1,1], \S):~
\gamma(-1) = -\varphi_1,~ 
\gamma(1) = \varphi_1\}, \\
\notag
\S &:= \{u \in W_0^{1,p}(\Omega):~ \|u\|_p=1\},
\end{align}
and $\varphi_1$ is the first eigenfunction of  \eqref{evp} with the normalization $\|\varphi_1\|_p = 1$. Note that $c(0)=\la_2(\Omega)$, the second eigenvalue of \eqref{evp}.
Using $c(t)$, the authors gave a description of the \textit{first nontrivial} curve $\mathscr{C}$ of the Fu\v{c}ik spectrum of \eqref{F} as the union of the points $(t+c(t), c(t))$, $t \geq 0$, and their reflections with respect to the diagonal $(t,t)$. Further,  they shown that $\mathscr{C}$ is continuous and each eigenfunction associated with a point on $\mathscr{C}$ has exactly two nodal domains (see Theorem 2.1 of \cite{Cuesta-Nodal}).

In \cite{Bartsch2005}, Bartsch et al.\ conjectured that in the linear case ($p=2$) any eigenfunction corresponding to a point on $\mathscr{C}$ is nonradial in a bounded radial domain (i.e., $\Omega$ is a ball or annulus). In the same article, they showed that the conjecture holds in a neighbourhood of $(\lambda_2(\Omega), \lambda_2(\Omega))$ (see Remark~5.2 of \cite{Bartsch2005}).
  A complete proof of this conjecture was given by Bartsch and Degiovanni in \cite{Bartsch} by estimating generalized Morse indices of corresponding eigenfunctions. In \cite{Benedikt2},
Benedikt et al.\  gave a different proof for this conjecture for a ball in $\R^N$ with $N=2$ and $N=3$. In this article,
we provide another proof for this conjecture for any bounded radial domain and  even extend this results  for general $p\in (1,\infty)$.
\begin{theorem}\label{FucikTHM}
 Let $p\in (1,\infty)$ and  $\Om$ be a  bounded radial domain in $\mathbb{R}^N$, $N \geq 2$. Then any  eigenfunction associated with a point on the first nontrivial curve $\mathscr{C}$ of the Fu\v{c}ik spectrum of the problem~\eqref{F} is nonradial.
\end{theorem}
We obtain the above result as a  simple consequence of Theorem~\ref{MainTHM}. Moreover, Theorem~\ref{FucikTHM} gives a generalization and a simpler proof for Theorem 1.1 of \cite{ADS-AMS} which states the nonradiality of second eigenfunctions of the $p$-Laplacian on a ball.

\section{Preliminaries}\label{section:prelim}

In this section, we first introduce the reflections with respect to the hyperplanes and the affine hyperplanes.  Then we briefly describe the shape derivative formula of \cite{Melian2001} and derive the formulas~\eqref{Form1} and \eqref{Form2} for $\la_1'(s)$.  Finally we state some results which will be required in the later parts of this article.   

For a nonzero vector $a \in \mathbb{R}^N,$ let $H_a$ be the hyperplane perpendicular to $a$, i.e.,
$$
H_{a} = \{x \in \R^N:~ \left<a, x \right> = 0\}. %
$$ 
Further, we define the half-spaces
$$
\mathcal{H}_a^+ := \{x \in \R^N:~ \left<a, x \right> > 0\}, \quad  \mathcal{H}_a^{-} := \{x \in \R^N:~ \left<a, x \right> < 0\}.
$$ 
Let $\sigma_a$ be the reflection with respect  to the hyperplane $H_a$, i.e.,  
\begin{align}\label{reflectiondef}
 \sigma_a(x)=x- 2\frac{\left<a, x \right>}{|a|^2}a =  x \left [I-2 \frac{a^Ta}{|a|^2}\right],\;\forall x \in \mathbb{R}^N,
\end{align}
where the last expression  is the matrix product of the vector $x$ and the matrix $\sigma_a =I-2 \frac{a^Ta}{|a|^2}$. 
Let $\widetilde{\sigma}_{a}$ be the reflection about the affine hyperplane  $se_1+H_{a}$.
Then $\widetilde{\sigma}_{a}$ is given as below:
$$
\widetilde{\sigma}_{a}(x)=x- 2\frac{\left<a, x-se_1 \right>}{|a|^2}a= \sigma_a(x)+2\frac{\left<a,se_1 \right>}{|a|^2}a.
$$
Now we recall the set $\Om_s= B_{R_1}(0)\setminus \overline{B_{R_0}(se_1)}$ and for each nonzero vector $a$ in $\R^N$, consider the following subsets of $\Om_s$:
\begin{align*}
\mathcal{O}^+_a &:=  \Om_s \cap \mathcal{H}_a^{+};\quad \mathcal{\widetilde{O}}^+_{a} :=  \Om_s \cap \left (\mathcal{H}_a^{+}+se_1\right );\\
\mathcal{O}^-_a &:=  \Om_s \cap \mathcal{H}_a^{-}; \quad  \mathcal{\widetilde{O}}^-_{a} :=  \Om_s \cap \left (\mathcal{H}_a^{-}+se_1\right ).
\end{align*}
The relation between some of the subsets of $\overline{\Om}_s$ under the reflections are listed below:
\begin{equation}\label{reflectionlist}
\left.
\begin{aligned}
           \sigma_{a}( \mathcal{O}^+_{a}) &=\mathcal{O}^-_{a}; \quad  \widetilde{\sigma}_{a}( \mathcal{\widetilde{O}}^+_{a}) =\mathcal{\widetilde{O}}^-_{a}, \; \forall\,  a\in \R^N\setminus\{0\} \mbox{ with}  \inpr{a,e_1}= 0;  \\
                      \sigma_{a}( \mathcal{O}^+_a) &\subset   \mathcal{O}^-_{a}; \quad \widetilde{\sigma}_{a}(\mathcal{\widetilde{O}}^+_{a}) \subset  \mathcal{\widetilde{O}}^-_{a},\; \forall\, a\in \R^N \mbox{ with}  \inpr{a,e_1}> 0; \\
             \sigma_{a}(\partial B_{R_0}(se_1)\cap \partial \mathcal{O}^+_a ) &\subset  \mathcal{O}^-_{a};\quad \widetilde{\sigma}_{a}(\partial B_{R_1}(0)\cap \partial \mathcal{\widetilde{O}}^+_a ) \subset  \mathcal{\widetilde{O}}^-_{a},\; \forall\, a\in \R^N \mbox{ with} \inpr{a,e_1}> 0;\\
       \sigma_{a}(\partial B_{R_1}(0)\cap \partial \mathcal{O}^+_a )& =  \partial B_{R_1}(0)\cap \partial \mathcal{O}^-_{a}, \; \forall\, a \in \R^N \setminus \{0\};\\ 
       \widetilde{\sigma}_{a}(\partial B_{R_0}(se_1)\cap \partial\mathcal{\widetilde{O}}^+_{a} ) &=  \partial B_{R_0}(se_1)\cap \partial\mathcal{\widetilde{O}}^-_{a}, \forall\, a \in \R^N \setminus \{0\}.
\end{aligned} 
\right\}
\end{equation}

Now for a function $u$ defined on $\overline{\Om}_s$ and for a vector $a\in \R^N\setminus\{0\}$  with $\inpr{a,e_1}\ge0$  we define two new functions $u_a: \overline{\mathcal{O}^+_a} \to \R$ and $\widetilde{u}_{a} : \overline{\mathcal{\widetilde{O}}^+_a} \to \R$ as below:
$$
{u}_{a}(x):= u(\sigma_a(x)); \quad \widetilde{u}_{a}(x):=u(\widetilde{\sigma}_{a}(x)).
$$ 
By recalling the notation $\sigma_a = I-2 \frac{a^Ta}{|a|^2}$ from \eqref{reflectiondef}, for  $u\in \C^1(\overline{\Om_s})$ we see that 
\begin{equation}\label{gradient}
 \Gr u_a(x)=\Gr u (\sigma_a(x))\sigma_a, \; \forall x\in \overline{\mathcal{O}^+_a}; \quad \Gr \widetilde{u}_a(x)=\Gr u (\widetilde{\sigma}_{a}(x))\sigma_a, \; \forall x\in \overline{\mathcal{\widetilde{O}}^+_a}.
\end{equation}
Further, the normal vector satisfies the following relations:
\begin{align}\label{normal}
 n(\sigma_a(x))=n(x) \sigma_a, \, \forall\,  x\in \partial B_{R_1}(0)\cap\mathcal{O}^+_a; \quad  n(\widetilde{\sigma}_{a}(x))=n(x) \sigma_a, \, \forall\,  x\in \partial B_{R_0}(se_1)\cap\mathcal{O}^+_a.
\end{align}
 
\subsection*{Shape derivative formulas}
  For a smooth bounded vector
field $V$ on $\R^N$  consider the perturbation of $\Om_s$ given as $\wide{\Om}_t= (I+tV)\Om_s.$  It is known by  Theorem 3 of \cite{Melian2001} that
$\la_1(t,V):=\la_1(\wide{\Om}_t)$ is differentiable at $t=0$ and the derivative  is
given by 
\begin{equation}\label{Shapeformula}
\la_1'(0,V):= \lim_{t\ra 0}  \frac{\la_1(t,V)-\la_1(0,V)}{t}=-(p-1)\int\limits_{\partial\Om_s}\abs{\frac{\pa u_s}{\pa
n}(x)}^p \left<V(x), n(x)\right> \, \dS,
\end{equation} 
where $n$ is the outward unit normal to $\partial\Om_s$ and ${u_s}$ is the first eigenfunction corresponding to
$\la_1(s)$ normalized as \begin{equation}\label{normalization}
                          u_s>0 \mbox{ and } \norm{u_s}_p=1.
                         \end{equation}
In \cite{Anisa2}, the authors considered the vector field $V$ as given   below:
\begin{equation}\label{eq:perturb}
V(x)=\rho(x)e_1, \ \rho \in \C_c^\infty(B_{R_1}(0)) \ \mbox{and} \ \rho(x)\equiv 1 \ \mbox{in a neighbourhood of} \ B_{R_0}(s e_1).
\end{equation}
For this choice of $V$ and  for $t$ sufficiently small, the perturbations $\wide{\Om}_t$ of $\Om_s$ are generated by the shifts of the inner ball. More precisely,  $$\wide{\Om}_t= \Om_{s+t}.$$  Therefore, one gets $\la_1(t,V)=\la_1(s+t), \la_1(0,V)=\la_1(s)$ and hence  \eqref{Shapeformula} yields
\begin{equation}\label{form1}
  \la_1'(s)= -(p-1)\int\limits_{\partial B_{R_0}(se_1)}\abs{\frac{\pa u_s}{\pa n}(x)}^p n_1(x)\, \dS,
\end{equation}
where $n_1$ is the first component of $n,$ the outward unit normal to $\pa\Om_s$ on $\partial B_{R_0}(se_1)$ (i.e., the inward unit normal to $\pa B_{R_0}(se_1)$).

To derive the expression \eqref{Form2} for $\la'(s)$ (i.e.,  formula involving the normal derivative of $u_s$ on the outer boundary),  we consider the perturbations of $\Om_s$ generated 
by the shifts of the outer boundary. Indeed, such perturbations can be obtained by taking a vector field $V(x)=-\rho(x)e_1$ with $\rho \in \C^\infty(\R^N)$ and 
\begin{enumerate}[(i)]
	\item $\rho = 0$ in a neighbourhood of the inner sphere $\partial B_{R_0}(se_1)$;
	\item $\rho = 1$ in a neighbourhood of the outer sphere $\partial B_{R_1}(0)$.
\end{enumerate}
For this choice of $V,$  for $t$ sufficiently close to 0,  observe that 
$$
\wide{\Om}_t= B_{R_1}(-te_1)\setminus \overline{B_{R_0}(se_1)}.
$$
From the translation invariance of the $p$-Laplacian, we get
$$
\la_1(t,V)=\la_1\left(B_{R_1}(0)\setminus \overline{B_{R_0}((s+t)e_1)}\right)=\la_1(s+t).
$$ 
Now  \eqref{Shapeformula} yields
\begin{equation}\label{form2}
\la_1'(s)=\lim_{t\ra 0}  \frac{\la_1(s+t)-\la_1(t)}{t}=(p-1) \int\limits_{\partial B_{R_1}(0)} \abs{\frac{\pa u_s}{\pa
n}(x)}^p n_1(x)\, \dS,
\end{equation}
where $n_1$ is the first component of $n,$ the outward unit normal to $\pa\Om_s$ on $\partial B_{R_1}(0)$ (i.e., the outward unit normal to $\pa B_{R_1}(0)$). 

Next we rewrite the integral in \eqref{form2} using certain symmetries of the domain $\Om_s.$ Set $u=u_s$ in \eqref{form2} and  express the integral as a sum of two integrals:
\begin{equation}\label{split1}
 \int\limits_{\partial B_{R_1}(0)} \abs{\frac{\pa u}{\pa
n}(x)}^p n_1(x)\, \dS= \int\limits_{\partial B_{R_1}(0) \cap \partial\mathcal{O}^+_{e_1}} \left|
\frac{\partial {u}}{\partial n}(x)
\right|^p n_1(x) \, \dS + \int\limits_{\partial B_{R_1}(0) \cap \partial\mathcal{O}^-_{e_1}} \left|
\frac{\partial {u}}{\partial n}(x)
\right|^p n_1(x) \, \dS.
\end{equation}
From \eqref{gradient} and \eqref{normal} we have $\frac{\partial {u}}{\partial n}(x') = \frac{\partial {u_{e_1}}}{\partial n}(x)$ and $n_1(x')=-n_1(x)$ on $\partial B_{R_1}(0)\cap\mathcal{O}^+_{e_1},$ where $x' = \sigma_{e_1}(x).$ Hence, we modify the second integral as below:
\begin{align}\label{split2}
\int\limits_{\partial B_{R_1}(0) \cap \partial\mathcal{O}^-_{e_1}} \left|
\frac{\partial {u}}{\partial n}(x)
\right|^p n_1(x) \, \dS 
=&  \int\limits_{\partial B_{R_1}(0) \cap \partial\mathcal{O}^+_{e_1}} \left|
\frac{\partial u}{\partial n}(x')
\right|^p n_1(x') \, \dS \no \\
&= -\int\limits_{\partial B_{R_1}(0) \cap \partial\mathcal{O}^+_{e_1}}  \left|
\frac{\partial {u_{e_1}}}{\partial n}(x)
\right|^p  n_1(x) \, \dS.
\end{align}
Thus, by combining \eqref{form2}, \eqref{split1} and \eqref{split2} we  get
\begin{equation}\label{Formnew1}
 \la_1'(s)=(p-1)\int\limits_{\partial B_{R_1}(0) \cap \partial\mathcal{O}^+_{e_1}} \left (\left|
\frac{\partial u}{\partial n}
\right|^p- \left|
\frac{\partial u_{e_1}}{\partial n}
\right|^p \right) n_1 \, \dS.
\end{equation}
Similarly we can rewrite  formula \eqref{form1}  as below:
\begin{equation}\label{Formnew2}
\la_1'(s)=-(p-1)\int\limits_{\partial B_{R_0}(se_1) \cap \partial\mathcal{\widetilde{O}}^+_{e_1}} \left (\left|
\frac{\partial u}{\partial n}
\right|^p- \left|
\frac{\partial {\widetilde{u}_{e_1}}}{\partial n}
\right|^p \right) n_1 \, \dS.
\end{equation}

\subsection*{Auxiliary results}
Next we state a few results that we require in the subsequent sections. 
First we recall some results about the regularity of eigenfunctions of \eqref{evp} (cf. Theorem 1.3 of \cite{Barles}).
\begin{proposition}\label{regularity}
	Let $\Om$ be a smooth domain in $\R^N$ and let $u$ be a first eigenfunction of \eqref{evp}. 
	Then the following assertions are satisfied.
	\begin{enumerate}
		\item[\normalfont{(i)}] $u\in \C^1(\overline{\Om}).$
		\item[\normalfont{(ii)}] $u\in \C^2(\overline{\Om_\de}),$ where $\Om_\de :=\{x \in \Omega:\, \mathrm{dist}(x,\partial \Omega) < \de\}$ and $|\nabla u| > m > 0$ in $\Om_\de$ for some $m$.
	\end{enumerate}
\end{proposition}

The following version of the strong maximum principle is  due to Vazquez (Section 4, \cite{vazquez1984}).   
\begin{proposition}\label{smp}
 Let $\Om$ be a domain in $\R^N.$  Let $w \in \C^1(\overline{\Om})$ be a positive function satisfying 
\begin{equation*}
 -{\rm div}\left(a_{ij}(x) \frac{\partial w}{\partial x_j} \right)  \ge 0  \mbox{ in } \Om,
\end{equation*} 
where $a_{ij}\in W^{1,\infty}_{loc}(\Om)$ and there exists $\al>0$ such that $a_{ij}(x) \xi_i \xi_j \ge \al |\xi|^2, \,\forall\, \xi \in \R^N\setminus\{0\}, \,\forall \, x\in \Om. $ Then 
\begin{enumerate}
   \item[\normalfont{(i)}] $w\equiv 0$ in $\Om$ or else $w>0$ in $\Om.$
   \item[\normalfont{(ii)}] Let $x_0$ be a point on $\partial \Om $ satisfying the interior sphere condition. If $w>0$ in $\Om$ and $w(x_0)=0$, then 
   $$\frac{\partial w}{\partial n}(x_0) <0,$$
   where $n$ is the outward unit normal to $\partial \Omega$ at $x_0.$
\end{enumerate}
\end{proposition}
In the next proposition we state a weak comparison result, see Theorem~2.1 and Proposition~4.1 of \cite{chorwadwala2015faber}.
\begin{proposition}\label{wcp}
Let $\Om$ be a domain in $\R^N$  with Lipschitz boundary. Let $u_1,u_2\in \C^1(\overline{\Om})$ be positive weak solutions of $-\De_p u = \la u^{p-1} \mbox{ in } \Om$. If  $u_1\ge u_2$ on $\partial \Om,$ then 
 $$
 u_1\ge u_2 \mbox{ in } \Om \mbox{ and }  \frac{\partial u_1}{\partial n} \le \frac{\partial u_2}{\partial n} \mbox{ on } \{x \in \partial \Om:\,u_1(x)=u_2(x)=0\}.
 $$
\end{proposition}

 \section{Main result}\label{section:radial}
In this section we give the proof of Theorem \ref{MainTHM}. 
We will be considering various annular regions apart from $\Om_s$, for simplicity we denote them as
$$
A_{r_1,r_0}(x,y)=B_{r_1}(x)\setminus \overline{B_{r_0}(y)}.
$$
In particular, $A_{R_1, R_0}(0, s e_1)=\Om_s.$ Throughout this section, unless otherwise specified, the eigenfunction $u_s$ is  the first eigenfunction of $-\De_p$ on  $\Om_s$ normalized as in \eqref{normalization}, namely 
$u_s>0 \mbox{ and } \norm{u_s}_p=1.$

The following result is proved in \cite{Anisa2} (see Theorem 3.1) using  formula \eqref{Formnew2}. Here, for the sake of completeness, we  present a proof by making use of  formula \eqref{Formnew1}. 
\begin{lemma}
 Let $s\in [0, R_1-R_0)$ and let $\la_1(s)$ be the first eigenvalue of $-\De_p$ on  $\Om_s$. Then $\la'(s)\le 0.$
\end{lemma}
\begin{proof}
 By setting $u=u_s$ and  noting  that $ \sigma_{e_1}(\mathcal{O}^+_{e_1}) \subset \mathcal{O}^-_{e_1}$ and $\sigma_{e_1}(\partial B_{R_0}(se_1) \cap \partial\mathcal{O}^+_{e_1})\subset\mathcal{O}^-_{e_1} $, we easily see that  $u_{e_1}$ and $u$  weakly satisfy the following problems:
\begin{equation*}
\begin{aligned}
-\Delta_p u_{e_1} &= \lambda_1(s) \, u_{e_1}^{p-1}, \\
u_{e_1} &= 0, \\
u_{e_1} &= {u},\\
u_{e_1} &>0,
\end{aligned}
 \quad 
\begin{aligned}
-\Delta_p u &= \lambda_1(s) \, u^{p-1} &&\text{ in }\mathcal{O}^+_{e_1}, \\
u &= 0 &&\text{ on } \partial B_{R_1}(0) \cap \partial\mathcal{O}^+_{e_1}, \\
u &= {u_{e_1}} &&\text{ on } H_{e_1} \cap \partial\mathcal{O}^+_{e_1},\\
u &=0 &&\text{ on } \partial B_{R_0}(se_1) \cap \pa\mathcal{O}^+_{e_1}.
\end{aligned}
\end{equation*}
Thus by applying the weak comparison principle (Proposition \ref{wcp}) we obtain ${u}_{e_1} \geq u$ in $\mathcal{O}^+_{e_1}$. Moreover, as $u = 0$ on $\partial B_{R_1}(0) \cap \partial 
\mathcal{O}^+_{e_1},$  Proposition \ref{smp} yields
\begin{equation}\label{Hopf}
\frac{\partial {u}_{e_1}}{\pa n} \leq 
\frac{\partial u}{\pa n} < 0
\text{ on } \partial B_{R_1}(0) \cap \partial\mathcal{O}^+_{e_1}.
\end{equation}
Now since $n_1(x)$  is positive  for $x \in \partial B_{R_1}(0) \cap \partial\mathcal{O}^+_{e_1}$, from \eqref{Formnew1} and \eqref{Hopf} we derive that 
 \begin{equation*}\label{leq0}
  \lambda_1'(s)=(p-1)\int\limits_{\partial B_{R_1}(0) \cap \partial\mathcal{O}^+_{e_1}} \left (\left|
\frac{\partial u}{\pa n}
\right|^p- \left|
\frac{\partial u_{e_1}}{\pa n}
\right|^p \right) n_1 \, \dS\;\leq 0.
 \end{equation*}
This completes the proof.
\end{proof}

\subsection*{Symmetries with respect to the hyperplanes}
First we study symmetries of the first eigenfunction of $-\De_p$ on $\Om_s$.  We show that for  $s\in (0, R_1-R_0)$ the associated first eigenfunction is symmetric with respect to the  hyperplanes perpendicular to $H_{e_1}.$

\begin{lemma}\label{symmetry_nonvert}
Let $s\in (0, R_1-R_0)$ and let $u_s$ be the first eigenfunction of $-\De_p$ on $\Om_s.$ If  $a\in \R^N\setminus \{0\}$ with $\inpr{a,e_1}=0,$ then  
$$
u_s(x)=u_s(\sigma_{a}(x)), \, \forall x \in \Om_s.
$$
In particular, for $i=2,3,\dots,N$
$$
u_s(x)=u_s(\sigma_{e_i}(x))=u_s(x_1,x_2,\dots,x_{i-1},-x_i,x_{i+1},\dots,x_N), \, \forall x\in \Om_s.
$$
\end{lemma}
\begin{proof}
Clearly for $a \neq 0$ with $\inpr{a,e_1}=0$,  $\mathcal{O}^+_{a}=\sigma_{a}(\mathcal{O}^-_{a})$ (see \eqref{reflectionlist}). Thus $u:=u_s$ and $u_a:= u_s \circ \sigma_a$ weakly satisfy the following problems, respectively:
\begin{equation*}\label{eq2}
\begin{aligned}
-\Delta_p u_a &= \lambda_1(s) \, u_a^{p-1}, \\
u_a &= u,
\end{aligned}
\quad 
\begin{aligned}
-\Delta_p u&=\la_1(s)\,u^{p-1}  &&\text{ in }\mathcal{O}^+_a,\\
u &= u_a \qquad &&\text{ on } \partial\mathcal{O}^+_a.
\end{aligned}
\end{equation*}
Now by the weak comparison principle (Proposition \ref{wcp}), we obtain that ${u}_a \equiv u$ in $\mathcal{O}^+_{a}$, which implies the desired assertions. 
\end{proof}

In the next lemma we show that $u_s$ is  symmetric also with respect to $H_{e_1}$  in  a neighbourhood of the outer boundary, provided  $\la_1'(s)=0.$  
\begin{lemma}\label{symmetry_vert}
  If  $\lambda_1'(s)=0$ for some $s\in(0, R_1-R_0)$, then there exists $r_1>0$ such that 
 $$
 u_s(x)=u_s(\sigma_{e_1}(x)),\, \forall x\in  A_{R_1, r_1}(0, 0).
 $$
\end{lemma}
\begin{proof}
We set $u=u_s$. Since  $u\in \C^1(\overline{\Om_s})$, $u>0$ and $u$ vanishes on $\partial B_{R_1}(0)$ and $\partial B_{R_0}(se_1),$  there exists  $r^*\in(R_0+s,R_1)$ such that $\frac{\partial u}{\pa x_1}(r^* e_1)=0$.
Define 
\begin{equation}\label{eq:r1}
r_1 = \sup \left\{\abs{x} > 0:\, \left<\Gr u(x),x\right>=0 \right\}.
\end{equation}
As $\frac{\partial u}{\pa n}(x)<0$ on $\partial B_{R_1}(0)$ (by Proposition \ref{smp}), $\inpr{\Gr u(x),x}<0$ in a neighbourhood of $\partial B_{R_1}(0).$ Thus clearly $r_1 \in [r^*,R_1).$
By the construction, $A_{R_1, r_1}(0,0)$ is the maximal annular neighbourhood of $\partial B_{R_1}(0)$ on which $\left<\Gr u(x),x\right>$ is nonvanishing.  Further, by the continuity of $\Gr u$ there must
exist $x_1 \in \partial B_{r_1}(0)$ such that  
\begin{equation}\label{x_1}
  \left<\Gr u(x_1),x_1\right>=0.
\end{equation} 
 Set $u_{e_1}=u\circ \sigma_{e_1}$ on  $ A_{R_1, r_1}(0, 0)\cap \mathcal{O}_{e_1}^+.$ To show  $u\equiv u_{e_1}$ we linearise the $p$-Laplacian on the domain $ A_{R_1, r}(0, 0)\cap \mathcal{O}_{e_1}^+$  with $r_1<r<R_1$ by setting $w= {u}_{e_1} - u$. Then $w$ weakly satisfies the following problem:
\begin{equation}
\label{eq:oper}
\begin{aligned}
- \mbox{div} (A(x) \nabla w) &= \lambda\left({u}_{e_1}^{p-1} - u^{p-1} \right) \ge 0 &&\mbox{ in } A_{R_1, r}(0,0) \cap\mathcal{O}^+_{e_1}, \\ \no
w &\geq 0   &&\mbox{ on }\partial (A_{R_1, r}(0,0) \cap\mathcal{O}^+_{e_1}),
\end{aligned}
\end{equation}
where the coefficient matrix  $A(x)=[a_{ij}(x)]$ is given by
\begin{align*}
a_{ij}(x) &= \int\limits_0^1 |(1-t) \nabla u(x) + t \nabla  u_{e_1}(x)|^{p-2} \\
&\times  \left[ I + (p-2) \frac{\left [(1-t) \Gr  u (x) + t \Gr 
u_{e_1}(x)]^T[(1-t)\Gr  u (x)  + t \Gr u_{e_1}(x) \right ]}{|(1-t) 
\nabla  u(x) + t \nabla u_{e_1}(x)|^2} \right]_{ij} \dt.
\end{align*}
Now we show that $A(x)$ is uniformly positive definite on $ A_{R_1, r}(0, 0)\cap \mathcal{O}_{e_1}^+.$ Since $\left<\Gr u(x),x\right>$ does not vanish on $A_{R_1, r_1}(0,0)$ and is negative near the boundary  $\partial B_{R_1}(0)$, we see that $\left<\Gr u(x),x\right>< 0$   in   $A_{R_1, r}(0, 0).$ By the continuity,  we can find $\de_r>0$ such that 
\begin{equation*}
  \left<\Gr u(x),x\right>< -\de_r \text{ in }  A_{R_1, r}(0, 0).
\end{equation*}
Notice that 
$\left<\Gr u_{e_1}(x),x\right>=  \left<\Gr (u(\sigma_{e_1}(x))),x\right>= \left<\Gr u(\sigma_{e_1}(x))\sigma_{e_1},x\right>=\left<\Gr u (\sigma_{e_1}(x)),\sigma_{e_1}(x)\right>.$ Thus, by the above inequality we have $ \left<\Gr u_{e_1}(x),x\right>< -\de_r \text{ in }  A_{R_1, r}(0, 0) \cap \mathcal{O}^+_{e_1}.$ Therefore,
\begin{align*} 
(1-t) \left<\Gr u(x), x\right> + t \left<\Gr u_{e_1}(x),x\right> < -\de_r, \forall t \in [0,1], \forall  x \in A_{R_1, r}(0,0) \cap\mathcal{O}^+_{e_1}.
\end{align*}
Hence, for $x\in A_{R_1, r}(0,0)$ we get
\begin{equation}\label{elliptic:eqn1}
\left|(1-t) \nabla u(x) + t \nabla {u}_{e_1}(x) \right|\ge \left| \left< (1-t) \nabla u(x) + t \nabla {u}_{e_1}(x), \frac{x}{|x|}\right> \right|> \frac{\de_r}{R_1}=m_r.
\end{equation}
Further, since $|\Gr u|$ is bounded in $A_{R_1, r}(0,0)$, there exists $M_r>0$ such that  
\begin{equation}\label{elliptic:eqn2}
 \left|(1-t) \nabla u(x) + t \nabla {u}_{e_1}(x) \right|\le M_r.                                                                                             
\end{equation}                                                                                            
Note that for each $a\in \R^N\setminus \{0\},$ the matrix $a^T a$ has eigenvalues $\{0,|a|^2\}.$ Thus, for any $y\in \R^N,$
\begin{equation}\label{elliptic:eqn3}
 \min\{1,p-1\}|a|^{p-2}|y|^2\le \inpr{ |a|^{p-2}\left[ I+(p-2)\frac{a^Ta}{|a|^2} \right ]y,y} \le \max\{1,p-1\}|a|^{p-2}|y|^2.
\end{equation}
From \eqref{elliptic:eqn1}, \eqref{elliptic:eqn2} and \eqref{elliptic:eqn3}, for $x\in A_{R_1, r}(0,0)$ and $y\in \R^N$  we obtain
$$ 
\inpr{A(x)y,y} \ge 
\left\{ 
\begin{array}{ll}
	m_r^{p-2}|y|^2 &\mbox{ for } p \ge 2, \\
	 (p-1)M_r^{p-2}|y|^2 &\mbox{ for } 1<p<2.
\end{array} 
\right. 
$$
Thus the differential operator in \eqref{eq:oper} defined by means of $A(x)$ is uniformly elliptic in $A_{R_1, r}(0,0).$ Moreover, by Proposition~\ref{regularity}, $a_{i j} \in \C^{1}({A_{R_1, r}(0,0)})$. Hence, the strong maximum principle for~\eqref{eq:oper} (Proposition~\ref{smp}) implies that either $w \equiv 0$, or $w>0$ in $A_{R_1, r}(0,0) \cap\mathcal{O}^+_{e_1}$. 
Moreover, if $w>0$ in $A_{R_1, r}(0,0) \cap\mathcal{O}^+_{e_1}$, then
\begin{equation*}\label{eq:tileu<u}
\frac{\partial {u}_{e_1}}{\partial n}- 
\frac{\partial u}{\partial n}=\frac{\partial w }{\partial n} < 0 
\text{ on } \partial B_{R_1}(0) \cap \partial\mathcal{O}_{e_1}.
\end{equation*}
Now  \eqref{Formnew1} together with the above inequality  implies that $\lambda_1'(s) < 0$, which contradicts our assumption $\lambda_1'(s) = 0$. 
Thus we must have $w\equiv 0$  and hence $u\equiv u_{e_1}$  in $A_{R_1, r}(0,0) \cap\mathcal{O}^+_{e_1}$. 
Since $r\in (r_1,R_1)$ is arbitrary, we conclude that $u(x)=u(\sigma_{e_1}(x)),\;\forall x\in A_{R_1, r_1}(0,0)$.
\end{proof}

Next we show that $u$ is symmetric in $A_{R_1, r_1}(0,0)$  with respect to all the hyperplanes.
\begin{lemma}\label{All-symmetry}
Let $s$ and $r_1$ be as in  Lemma~\ref{symmetry_vert}. Then for any nonzero vector $a\in \R^N$ 
$$
u_s(x)=u_s(\sigma_a(x)),\, \forall x\in A_{R_1, r_1}(0,0).
$$

\end{lemma}
\begin{proof}
The case $\inpr{a,e_1}=0$ follows from Lemma \ref{symmetry_nonvert}. Note that $\sigma_{a}(x)=\sigma_{k a}(x)$ for $k\in \R\setminus \{0\}.$ Thus, it is enough to prove the result for $a \in A_{R_1,r_1}(0,0)$ with $\inpr{a,e_1}>0.$ In this case we have $\sigma_a(\mathcal{O}^+_{a})\subset\mathcal{O}^-_{a}$. Now by setting $u=u_s$ and $u_a= u_s \circ \sigma_a$ we see that  $u_a$ and $u$ satisfy the following problems in $\mathcal{O}^+_{a}$:
\begin{equation*}
\begin{aligned}
-\Delta_p u_a &= \lambda_1(s) \, u_a^{p-1},\\
u_a &= 0,\\
u_a &= {u},\\
u_a &>0,
\end{aligned}
\quad 
\begin{aligned}
-\Delta_p u &= \lambda_1(s) \, u^{p-1} &&\text{ in }\mathcal{O}^+_a, \\
u &= 0 &&\text{ on } \partial B_{R_1}(0) \cap \partial\mathcal{O}^+_a, \\
u &= {u_a} &&\text{ on } H_a \cap \partial\mathcal{O}^+_a,\\
u &=0 &&\text{ on } \partial B_{R_0}(se_1) \cap \partial\mathcal{O}^+_a.
\end{aligned}
\end{equation*}
Applying  the weak comparison principle (Proposition~\ref{wcp}), we obtain that ${u}_a \geq u$ in $\mathcal{O}^+_{a}$. As before we set $w=u_a-u.$ From Lemma \ref{symmetry_nonvert} and Lemma \ref{symmetry_vert} we obtain  $u(a) = u(-a)$ as below:
\begin{align*}
u(a_1, a_2, \dots, a_N)&=
u(a_1, -a_2, \dots, a_N) \\
&=\dots = 
u(a_1, -a_2, \dots, -a_N) = 
u(-a_1, -a_2, \dots, - a_N).
\end{align*}
By definition $u_a(a)=u(-a)$ and hence $w(a)=0$. Now we proceed along the same lines as in Lemma \ref{symmetry_vert} and see that $w$ satisfies the following problem:
\begin{align*}
 -{\rm div}(A(x)w)\ge 0 \mbox{ in } A_{R_1,r}(0,0)\cap\mathcal{O}^+_{a}  ;\quad
 w\ge 0 \mbox{ on } \partial(A_{R_1,r}(0,0)\cap \mathcal{O}^+_{a})
\end{align*}
for any $r \in (r_1, R_1)$, where the coefficient matrix $A(x)$ is uniformly positive definite. By the strong maximum principle we have either $w\equiv 0$ or else $w>0$ in $A_{R_1,r}(0,0) \cap\mathcal{O}^+_{a}.$ Since $w(a)=0,$ we obtain  $w\equiv 0$ and hence $u\equiv u_a$ in $A_{R_1,r}(0,0) \cap\mathcal{O}^+_{a}.$ Finally, using the reflection, we conclude that $u(x)=u(\sigma_{a}(x)),$ $\forall\, x\in A_{R_1, r_1}(0,0).$
\end{proof}

\begin{theorem}\label{radial-outer}
	 Let $s\in (0, R_1-R_0)$ and let $u_s$ be the first eigenfunction  of $-\De_p$ on $\Om_s.$ If $\la_1'(s)=0$, then  $u_s$ is radial in the annulus $A_{R_1, r_1}(0,0),$ where $r_1$ is given by Lemma~\ref{symmetry_vert}. 
	 Furthermore,  $\nabla u_s= 0$ on $\partial B_{r_1}(0)$. 
\end{theorem}
\begin{proof}
Let $b,c \in A_{R_1,r_1}(0,0)$ be such that $b\ne c$ and $|b|=|c|.$ Then  there exists a constant $k$ such that $a = k(b-c)\in A_{R_1,r_1}(0,0).$ Noting that $\sigma_a(b)=c,$ from Lemma~\ref{All-symmetry} we obtain that 
$$
u_s(b)=u_s(\sigma_a(b))=u_s(c).
$$ 
Since $b$ and $c$ are arbitrary, we conclude that $u_s$ is radial in the annulus $A_{R_1,r_1}(0,0)$.
Further, as $u_s$ is continuously differentiable in $A_{R_1,r_1}(0,0)$ and $\Gr u_s (x_1)\cdot x_1=0$ (see \eqref{x_1}), the radiality of $u_s$ gives $\nabla u_s  = 0$ on $\partial B_{r_1}(0)$. 
\end{proof}

\subsection*{Symmetries with respect to the affine hyperplanes passing through $se_1$}
In this subsection  we prove the radiality (up to a translation of the origin) of $u_s$ in a neighbourhood of the inner boundary. Since $\widetilde{\sigma}_{a}(x)={\sigma}_{a}(x)$ for $a$ such that $\inpr{a,e_1}=0$, 
Lemma~\ref{symmetry_nonvert} holds as it is, and hence we have for $i=2,\dots,N$
$$
u_s(x)=u_s(\widetilde{\sigma}_{e_i}(x))=u(x_1,x_2,\dots,x_{i-1},-x_i,x_{i+1},\dots,x_N),\, \forall x\in \Om_s.
$$
Next we prove a symmetry result along the same lines as in Lemma~\ref{symmetry_vert}. 
\begin{lemma}\label{symmetry-affine-vertical}
 Let $s \in (0, R_1-R_0)$ and let $u_s$ be the first eigenfunction of  $-\De_p$ on $\Om_s.$ If $\lambda_1'(s)=0$, then there exists $r_0>0$ such that 
 $$u_s(x)=u(\widetilde{\sigma}_{e_1}(x))= u_s(-x_1+2s,x_2,\dots,x_N) ,\, \forall x\in  A_{r_0, R_0}(se_1, se_1) .$$
\end{lemma}
\begin{proof}
As it was shown in the proof of Lemma \ref{symmetry_vert}, we have  $r^*\in(R_0+s,R_1)$ such that $\frac{\partial u}{\pa x_1}(r^* e_1)=0$.
Define 
\begin{equation}\label{eq:r0}
r_0 = \inf \left\{\abs{x-se_1} > 0:\, \left<\Gr u (x),x-se_1\right>=0 \right\}.
\end{equation}
Clearly   $r_0 \in(R_0,R_1-s),$ since by Hopf's maximum principle  $\left<\Gr u (x),x-se_1\right>=|x-se_1|\frac{\partial u}{\pa n}(x)\ne 0$ on $\partial B_{R_0}(se_1)$.
By the construction, $A_{r_0, R_0}(se_1,se_1)$ is the maximal annular neighbourhood of $\partial B_{R_0}(se_1)$ on which  $\left<\Gr u (x),x-se_1\right>$ is nonvanishing. Further, by the continuity of $\Gr u$ there must
 exist $x_0 \in \partial B_{r_0}(se_1)$ such that  
\begin{equation*}\label{x_0}
	\left<\Gr u (x_0),x_0-se_1\right>=0.
\end{equation*} 
As in the proof of Lemma \ref{symmetry_vert}, we linearise the $p$-Laplacian on the domain $A_{r, R_0}(se_1,se_1)\cap \widetilde{\mathcal{O}}_{e_1}^+$ with $R_0<r<r_0$ by setting $w= {\widetilde{u}}_{e_1} - u$. Then  $w$ weakly satisfies the following problem:
\begin{equation*}
\begin{aligned}
- \mbox{div} (A(x) \nabla w) &= \lambda\left({\widetilde{u}}_{e_1}^{p-1} - u^{p-1} \right)\ge 0  &&\mbox{ in } A_{r, R_0}(se_1,se_1) \cap \widetilde{\mathcal{O}}_{e_1}^+, \\
w &\geq 0   &&\mbox{ on }\partial (A_{r, R_0}(se_1,se_1)\cap \widetilde{\mathcal{O}}_{e_1}^+).
\end{aligned}
\end{equation*}
By  similar arguments as in Lemma \ref{symmetry_vert}, the above differential operator is  uniformly elliptic on $A_{r, R_0}(se_1,se_1) \cap \widetilde{\mathcal{O}}_{e_1}^+$ and hence by the strong maximum principle we  have either $w\equiv 0$ or $w>0$ on this domain. 
If $w>0$ in $A_{r, R_0}(se_1,se_1) \cap \widetilde{\mathcal{O}}_{e_1}^+$, then by the Hopf maximum principle 
\begin{equation*}\label{normal2}
\frac{\partial {\widetilde{u}}_{e_1}}{\partial n}- 
\frac{\partial u}{\partial n}=\frac{\partial w }{\partial n} < 0 
\text{ on } \partial B_{R_0}(se_1) \cap \partial\widetilde{\mathcal{O}}_{e_1}^+.
\end{equation*}
Now  \eqref{Formnew2} implies that $\lambda_1'(s) < 0$, a contradiction to the assumption $\lambda_1'(s) = 0$. Thus we must have $w\equiv 0$ and hence $u\equiv \widetilde{u}_{e_1}$ in $A_{r, R_0}(se_1,se_1) \cap \widetilde{\mathcal{O}}_{e_1}^+.$ Since $r\in (R_0,r_0)$ is arbitrary, we obtain the desired fact.
\end{proof}

Next we state a lemma which is a counterpart of Lemma \ref{All-symmetry}. The proof follows along the same lines.
\begin{lemma}\label{All-symmetry-affine}
Let $s \in (0, R_1-R_0)$ and let $u_s$ be the first eigenfunction  of $-\De_p$ on $\Om_s.$ If $\la_1'(s)=0$, then for any nonzero vector $a\in \R^N$ 
$$
u_s(x)=u_s(\widetilde{\sigma}_a(x)),\, \forall x\in A_{r_0, R_1}(se_1,se_1),
$$
where $r_0$ is given by Lemma \ref{symmetry-affine-vertical}. 
\end{lemma}

The next theorem, which is a counterpart of Theorem \ref{radial-outer}, states that  $u_s$ is radial (up to a translation of the origin) in a neighbourhood of the inner ball. The proof follows  along the same lines using Lemma \ref{symmetry_nonvert}, Lemma \ref{symmetry-affine-vertical} and Lemma \ref{All-symmetry-affine}. 
\begin{theorem}\label{radial-inner}
 Let $s\in (0, R_1-R_0)$ and let $u_s$ be the first eigenfunction  of $-\De_p$ on $\Om_s.$ If $\la_1'(s)=0$, then  $u_s$ is radial in the annulus $A_{r_0,R_0}(se_1,se_1).$ Furthermore,  $\nabla u_s= 0$ on $\partial B_{r_0}(se_1)$. 
\end{theorem}
\begin{remark}\rm{
Let $u_0$ be a positive first eigenfunction of $-\De_p$ on $A_{R_1, R_0}(0, 0).$  Note that $u_0$ is radial (cf.~\cite[Proposition 1.1]{nazarov2000}) and one can verify that $u_0$ attains its maximum on a unique sphere of radius $\bar{r} \in (R_0, R_1)$ and $u_0'(\bar{r})=0.$ From the simiplicity of the first eigenvalue, it is clear that  every first eigenfunction $u$ of $-\De_p$ on $A_{R_1, R_0}(0, 0)$ is radial and $u'(\bar{r})=0.$}
\end{remark}
\begin{lemma}\label{lem:r0=r1}
Let $\lambda_1'(s)=0$ for some $s \in (0, R_1-R_0)$. Let $r_0$ and $r_1$ be given by Lemmas~\ref{symmetry_nonvert} and \ref{symmetry-affine-vertical}, respectively. Then $r_0=r_1=\bar{r}.$
\end{lemma}
\begin{proof}
 From the definitions of $r_0$ and $r_1$ (see \eqref{eq:r0} and \eqref{eq:r1}) it easily follows that $r_0\le r_1.$ First we show that $r_1 \leq \bar{r}$. Suppose that $r_1 > \bar{r}$. For notational simplicity, we denote an annular region with centre at the origin  as $A_{t_1,t_0}=A_{t_1,t_0}(0,0)$.  
 Now consider the following function on  $A_{R_1,R_0}$:
\begin{equation*}\label{w1}
w_1(x) = 
\left\{
\begin{aligned}
&u_{s}(x),	&&x \in A_{R_1, r_1},\\
& C_1, 		&&x \in A_{r_1, \bar{r}},\\
&u_{0}(x),	&&x \in A_{\bar{r}, R_0},
\end{aligned}
\right.
\end{equation*}
where $C_1=u_s(x)$ for $|x|=r_1.$ By multiplying with an appropriate constant we can choose $u_0$ in such a way that  $u_0(x)= C_1$ for $|x|=\bar{r}.$ 
Since  $w_1$ is continuous and piecewise differentiable on $A_{R_1, R_0}$ we have $w_1 \in W_0^{1,p}(A_{R_1, R_0})$.  To estimate $\norm{\Gr w_1}_p^p$, we derive
a few identities.  Note that for any $r\in (r_1,R_1),$ $\Gr u_s$ does not vanish on $A_{R_1,r}$ and hence $u_s \in \C^2(A_{R_1,r})$, see Proposition~\ref{regularity}. Thus $u_s \in \C^2(A_{R_1,r_1})$ and hence the 
following equation holds pointwise  in $ A_{R_1,r_1}:$
$$-\Delta_p u_{s} = \lambda_1(s) |u_{s}|^{p-2} u_s.$$ 
Multiply the above equation by $u_s$  and integrate over $A_{R_1,r_1}$ to get
$$
 \int\limits_{ A_{R_1,r_1}} -\Delta_p u_s\; u_s \dx  = \lambda_1(s) \int\limits_{ A_{R_1,r_1}}  |u_s|^{p-2} u_s\, u_s \dx. 
$$
Now by noting that $\Gr u_s =0$ on  $\partial B_{r_1}(0)$ and $u_s=0$ on  $\partial B_{R_1}(0)$, the integration by parts gives
\begin{align}\label{identity1}
\int\limits_{A_{R_1,r_1}} |\nabla u_s|^p \dx = \lambda_1(s) \int\limits_{A_{R_1,r_1}}
|u_s|^p \dx.
\end{align}
Similarly 
\begin{align}\label{identity2}
\int\limits_{A_{\bar{r},R_0}}  |\nabla u_{0}|^p \dx = \lambda_1(0) \int\limits_{A_{\bar{r},R_0}}  |u_{0}|^p \dx.
\end{align}
Now we estimate $\norm{\Gr w_1}_p^p$:
\begin{align*}
\int\limits_{A_{R_1,R_0}} |\nabla w_1|^p \dx & =  \int\limits_{A_{R_1,r_1}} |\nabla u_{s}|^p \dx +  \int\limits_{A_{\bar{r},R_0}}  |\nabla u_{0}|^p \dx.
\end{align*}
By using  \eqref{identity1} and \eqref{identity2} and inequality  $\la_1(s)\le \lambda_1(0)$ we  obtain 
$$
\int\limits_{A_{R_1,R_0}} |\nabla w_1|^p \dx \le \la_1(0) \left ( \int\limits_{A_{R_1,r_1}} |u_{s}|^p \dx + 
\int\limits_{A_{\bar{r},R_0}}  |u_{0}|^p \dx \right ).
$$
Next we estimate $\norm{w_1}_p^p$:
\begin{align*}
\displaystyle \int\limits_{A_{R_1,R_0}} |w_1|^p \dx 
&= \int\limits_{A_{R_1,r_1}} |u_{s}|^p \dx + 
\int\limits_{A_{r_1,\bar{r}}} C_1^p \dx +
\int\limits_{A_{\bar{r},R_0}}  |u_{0}|^p \dx\\
&> \int\limits_{A_{R_1,r_1}} |u_{s}|^p \dx + 
\int\limits_{A_{\bar{r},R_0}}  |u_{0}|^p \dx.
\end{align*}
Now combining the above estimates, we arrive at 
$$
\int\limits_{A_{R_1,R_0}} |\nabla w_1|^p \dx < \lambda_1(0) \int\limits_{A_{R_1,R_0}} |w_1|^p \dx,
$$
a contradiction to the definition of $\lambda_1(0)$. Hence we must have  $r_1 \le \bar{r}.$ \\
Next we show that $\bar{r}\le r_0.$  Suppose that $\bar{r}> r_0.$ In this case, we define $w_2$  on $A_{R_1,R_0}$  as below:
\begin{equation*}\label{w2}
w_2(x) = 
\left\{
\begin{aligned}
&u_{0}(x),	&&x \in A_{R_1,\bar{r} },\\
& C_2, 		&&x \in A_{\bar{r},r_0},\\
&u_s(x+se_1),	&&x \in A_{r_0, R_0},
\end{aligned}
\right.
\end{equation*}
where $C_2=u_s(x)$ for $|x+se_1|=r_0$ and $u_0$ is scaled to satisfy $u_0(x)=C_2$ for $|x|=\bar{r}.$ As before we see that 
$w_2\in W^{1,p}_0(A_{R_1,R_0})$ and 
$$\displaystyle 
\int\limits_{A_{R_1,R_0}} |\nabla w_2|^p \dx < 
\lambda_1(0)\int\limits_{A_{R_1,R_0}} |w_2|^p \dx,
$$
which again contradicts the definition of $\la_1(0).$ Thus $\bar{r}\le r_0$ and we conclude  that $r_0=\bar{r}=r_1.$  
\end{proof}
Now we give a proof of our main theorem. 

\medskip
\noindent
{\bf Proof of Theorem \ref{MainTHM}:}\\
 Suppose that there exists $s > 0$ such that $\lambda_1'(s)=0$. Now Lemmas~\ref{symmetry_nonvert}, \ref{symmetry-affine-vertical} and \ref{lem:r0=r1} give $r_0$ and $r_1$ with $r_0=r_1$.  Further, from the definitions of $r_0$ and $r_1$ (see \eqref{eq:r0} and \eqref{eq:r1}) we can deduce that
 $$
 \Gr u ((r_0+s)e_1)=0 \mbox{ and }  \Gr u (re_1)\neq 0,\, \forall  r>r_1.
 $$   
 This is a contradiction, since $r_0+s= r_1+s>r_1.$
 Thus $\la_1'(s)<0$ for all $s \in (0, R_1-R_0).$
 \qed
\begin{remark}\label{rem:mainthm}
 {\rm
 Note that in Theorem~\ref{MainTHM} we consider only the case $\overline{B_{R_0}(se_1)} \subset B_{R_1}(0)$, i.e., $s \in [0, R_1-R_0)$.  For any $s_1, s_2$ satisfying $\sqrt{R_1^2 - R_0^2} \leq s_1 < s_2 \leq R_1 + R_0$, it is geometrically evident that 
 $$
 B_{R_1}(0)\setminus \overline{B_{R_0}(s_1 e_1)} \subsetneq B_{R_1}(0)\setminus \overline{B_{R_0}(s_2 e_1)}.
 $$
 Now the strict domain monotonicity of $\lambda_1(s)$ (cf. Lemma~5.7 of \cite{Cuesta-Fucik}) gives   $\lambda_1(s_1) > \lambda_1(s_2).$ Thus  $\lambda_1(s)$ is strictly decreasing on $[\sqrt{R_1^2 - R_0^2}, R_1 + R_0]$.  Further,   $\lambda_1(s) = \lambda_1(B_{R_1}(0))$ for  $s > R_1+R_0$.
}
\end{remark}

\begin{remark} {\rm
 It can be easily seen that the measure of the set $ B_{R_1}(0)\setminus \overline{B_{R_0}(se_1)}$ strictly decreases with respect to $s \in [R_1-R_0, \sqrt{R_1^2 - R_0^2}].$ However, nothing is known about the behaviour of $\la_1(B_{R_1}(0)\setminus \overline{B_{R_0}(se_1)})$ on this interval.}
\end{remark}

\begin{remark} {\rm Let $\Om_0,\Om_1$ be any two balls  $\R^N$ such that 
$
 \Om_0\subsetneq \Om_1, |\Om_0|=|B_0| \text{ and } |\Om_1|=|B_1|,
$ where   $B_0$ and $B_1$ are concentric balls.
  Then Theorem \ref{MainTHM} gives us that
  $\la_1(\Om_1\setminus\overline{\Om_0})\le \la_1(B_1\setminus\overline{B_0}).$
  This inequality does not hold in general, if $\Om_0$ and $\Om_1$ are not balls. For example, consider the rectangular domains $\Om_0$ (sides $\frac{\pi R_0}{n}$ and $R_0 n$) and $\Om_1$ (sides $\frac{\pi R_1}{n}$ and $R_1 n$). Clearly $\la_1(\Om_1\setminus \Om_0)\ra \infty$ as $n\ra \infty$ and $\la_1(B_1\setminus \overline{B_0})=\la_1(A_{R_1,R_0}(0,0))<\infty.$}
 \end{remark}

 \section{Limit cases \texorpdfstring{$p=1$}{p=1} and \texorpdfstring{$p=\infty$}{p=infinity}}\label{section:limits}

In this section we prove Theorem~\ref{THM2}.  Recall that 
\begin{equation*}
\Lambda_\infty(s) := \lim\limits_{p \to \infty} \la_1^{1/p}(p,s)  
\quad \text{and} \quad 
\Lambda_1(s) := \lim\limits_{p \to 1} \la_1(p,s).
\end{equation*}
By Theorem~\ref{MainTHM}, for any $p>1$ and $0 \le s_1 < s_2 < R_1-R_0$ it holds that $0 < \lambda_1(p,s_2) < \lambda_1(p,s_1)$ and hence we immediately deduce that
\begin{equation}\label{ineq}
 0\le\Lambda_1(s_2) \le  \Lambda_1(s_1),\quad 0\le \Lambda_\infty(s_2) \le \Lambda_\infty(s_1).
\end{equation}
To show that $\Lambda_\infty(s)$ is continuous and strictly decreasing on $[0,R_1-R_0)$, we use the following geometric characterization of  $\Lambda_\infty(s)$ obtained in  \cite{juutinen1999}:
$$
\Lambda_\infty(s) = \frac{1}{r_\text{max}},
$$
where $r_\text{max}$ is the radius of a maximal ball inscribed in $\Om_s$.

\medskip
\noi{\bf Proof of part (i) of Theorem~\ref{THM2}}. For $s \in [0, R_1 - R_0),$ a simple calculation shows that $r_\text{max} = \frac{R_1 - R_0 + s}{2}$ and hence 
$$
\Lambda_\infty(s)=\frac{2}{R_1 - R_0 + s}.
$$
Thus one can  easily see that  $\Lambda_\infty(s)$  is continuous and strictly decreasing on  $s \in [0, R_1-R_0)$. 
\qed

\begin{remark}
 {\rm The geometric characterization of  $\Lambda_\infty(s)$ allows us to compute $\Lambda_\infty(s)$ even for  $s\ge R_1-R_0.$ Indeed, the same calculation gives us, 
 \begin{equation*}
\Lambda_\infty(s)= \left\{ \begin{array}{ll}
                               \frac{2}{R_1 - R_0 + s} \quad  &\text{for} \quad s \in [0, R_1 + R_0),\\
                               \frac{1}{R_1} \quad  &\text{for} \quad s \geq R_1 + R_0.
                               \end{array} \right.
    \end{equation*}
Clearly $\Lambda_\infty(s)$ is continuous everywhere and differentiable except at the points $s=0$ and $s = R_1 + R_0$. 
}
\end{remark}
We refer the reader to \cite{navarro2014} for related problems on the domain dependence of $\Lambda_\infty$.
\medskip

Now we consider the case $p=1$. From \eqref{ineq} we know that $\Lambda_1(s)$ is decreasing. 
To show the continuity of $\Lambda_1(s)$ and to prove part (ii) of Theorem \ref{THM2},  we use the following variational  characterization of  $\Lambda_1(s)$ given in  \cite{Kawohl-Fridman}:
$$
\Lambda_1(s) = h(s),
$$
where $h(s)$ stands for the Cheeger constant of $\Om_s$ which can be defined as
\begin{equation}\label{eq:cheeger}
h(s) := \inf \frac{|\partial D|}{|D|}.
\end{equation}
Here the infimum is taken over all Lipschitz subdomains $D$ of $\overline{\Om}_s$ and  $|\cdot|$  denote the Hausdorff measures (coincide with the usual volume and surface area for Lipschitz domains) of dimension $N-1$ in the numerator  and the dimension $N$ in the denominator. Any minimizer of \eqref{eq:cheeger} is called a Cheeger set. It is known that a Cheeger set always exists, see Theorem~8 of \cite{Kawohl-Fridman}. 

As in Section \ref{section:prelim}, by considering perturbations of $\Omega_s$ given by the vector field in \eqref{eq:perturb} we apply Theorem 1.1 of \cite{parini2015shape} to conclude that $h(s)$ is continuous on $[0, R_1-R_0)$.

\medskip
\noi{\bf Proof of part (ii) of Theorem \ref{THM2}}.
It is known (see, for instance, \cite{Ercole2015} and also references therein) that concentric annulus $\Om_0$ is calibrable, (i.e.,  $\Om_0$  itself is a Cheeger set of $\Om_0$) and hence
$$
h(0) = \frac{|\partial \Om_0 |}{|\Om_0|} = N\, \frac{R_1^{N-1} + R_0^{N-1}}{R_1^N - R_0^N}.
$$
On the other hand,  for the eccentric annulus $\Om_s$ with $s \in (0, R_1-R_0)$  it is clear that
$$
h(s) \leq \frac{|\partial \Om_s |}{|\Om_s|} = N\, \frac{R_1^{N-1} + R_0^{N-1}}{R_1^N - R_0^N} = h(0).
$$
Next we show that for $s$ sufficiently close to $R_1-R_0$ the above inequality is strict. For this we construct  an appropriate  subset $D$ of $\Om_s$  satisfying $\frac{|\partial D|}{|D|}<h(0)$. 

\begin{figure}[!ht]
	\begin{center}
		\includegraphics[width=0.6\linewidth]{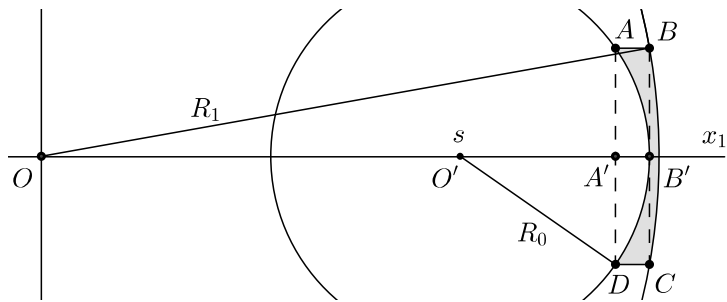}
		\caption{``Convex-concave lens'' $ABCD_{\text{lens}}$ (grey) and cylinder $ABCD_{\text{cyl}}$ (dashed).}
		\label{fig:eccentr1}
	\end{center}
\end{figure}
In this proof,  without any ambiguity,  we use $|\cdot|$ to denote the various measures such as the length, surface area and volume of the objects lie in the appropriate spaces. 
Let $\varepsilon > 0$ be sufficiently small and let $B' = |OB'| \, e_1$ be the point such that $|OB'|= \sqrt{R_1^2-\varepsilon^2}$ (see Fig.~\ref{fig:eccentr1}).
Then the hyperplane perpendicular to $e_1$ at $B'$ intersects with $B_{R_1}(0)$ by the $(N-1)$-dimensional ball $B_1$ of radius $|BB'| = \varepsilon$.
By choosing   $s = s_\varepsilon = \sqrt{R_1^2 - \varepsilon^2} - R_0$,  we see that the ball  $B_{R_0}(se_1)$ touches $B_1$.
Now consider the $N$-dimensional ``convex-concave lens'' $ABCD_{\text{lens}}$ bounded by the spherical caps $BC_{\text{cap}}$ and $AD_{\text{cap}}$ of the spheres $\partial B_{R_1}(0)$ and $\partial B_{R_0}(se_1)$, respectively, and by the lateral cylindrical surface $AB_{\text{lat}}$ generated by the segment $AB$ parallel to $e_1$. Let $ABCD_{\text{cyl}}$ be the  cylinder of radius $|BB'|$ and height $|AB|$.  For simplicity,  we denote the various positive constants which  are independent of $\varepsilon$   by  $k.$ 
For $\varepsilon > 0$ small enough, observe that 
\begin{align*}
|AB| & = |A'B'| = R_0 - \sqrt{R_0^2 - \varepsilon^2} \approx k\varepsilon^2; \\
|AD_{\text{cap}}| & > |BC_{\text{cap}}| > |B_1| = k \varepsilon^{N-1};\\
|ABCD_{\text{lens}}| & < |ABCD_{\text{cyl}}| = |AB| |B_1| \approx k \varepsilon^{2} \, \varepsilon^{N-1};\\ 
|AB_{\text{lat}}| &= |AB| |\partial B_1| \approx k \varepsilon^{2} \, \varepsilon^{N-2}.
\end{align*} 
Now by making  use of  the above estimates we obtain
\begin{align*}
\frac{|\partial\left(\Om_s \setminus ABCD_\text{lens}\right)|}{|\Om_s \setminus ABCD_\text{lens}|} = 
\frac{|\partial \Om_s| - |AD_{\text{cap}}| - |BC_{\text{cap}}| + |AB_{\text{lat}}|}{|\Om_s|-|ABCD_{\text{lens}}|}
< \frac{|\partial \Om_s| - 2k \varepsilon^{N-1} + k \varepsilon^{N}}{|\Om_s|-k \varepsilon^{N+1}} < 
\frac{|\partial \Om_s|}{|\Om_s|}
\end{align*}
for sufficiently small $\varepsilon$.  Therefore, there exists $s>0$ such that $h(s)<h(0).$ Now define
\begin{equation}\label{eq:s^*}
s^* := \inf\{s \in [0, R_1-R_0):~ h(0) > h(s)\}.
\end{equation}
 Since $h(s)$ is continuous,  by the definition of  $s^*,$ we easily see that  $h(0) = h(s^*)$   and $h(s^*) > h(s)$ for any $s \in (s^*, R_1-R_0)$.
\qed

\begin{remark}\label{rem:cheeger}
 {\rm 
 	Clearly $h(s)=h(0)$ for every $s\in [0,s^*].$ Thus, if $s^*>0$, then the strict monotonicity  of  $\lambda_1(s)$  fails for $p=1$. However,  whether $s^*>0$ or not is still open. 
 Further, the strict monotonicity of $h(s)$ on the interval $[s^*,R_1-R_0]$ is not answered yet. It is worth to mention that a shape derivative formula for $h_1(\Om)$ is obtained in \cite{parini2015shape} for  $\Om$ having just one Cheeger set. However, the uniqueness of the Cheeger set for eccentric annular regions $\Om_s$ is not known.
 }
\end{remark}

 \section{Application to the Fu\v{c}ik spectrum}\label{fucik}
In this section we prove Theorem \ref{FucikTHM}. For this end, we use Theorem \ref{MainTHM} and the variational characterization \eqref{eq:Fucik} of $\mathscr{C}$, the first nontrivial curve of the Fu\v{c}ik spectrum for the eigenvalue problem~\eqref{F}, see~\cite{Cuesta-Fucik}.
Recall that $\mathscr{C}$ is constructed from  points $(t+c(t), c(t))$, where
$$
c(t) = \inf_{\gamma \in \Gamma}
\max_{u \in \gamma[-1,1]} \left(\int\limits_{\Omega}|\nabla u|^p \, \dx - t \int\limits_{\Omega}(u^+)^p \, \dx
\right), \quad t \geq 0,
$$
and their reflections with respect to the diagonal. 
See \eqref{Gamma} for the definition of $\Gamma$.

\medskip
\noi{\bf Proof of Theorem \ref{FucikTHM}:} 
Let $\Om$ be a bounded radial domain. Suppose there exist a point on $\mathscr{C}$ and a corresponding eigenfunction $u$ which is radial. Without loss of generality, we can suppose that $t \geq 0$ (otherwise we consider $-u$ instead of $u$). 
Thus $u$ satisfies the following equation:
 \begin{equation*}
\left.
\begin{aligned}
-\Delta_p u &=(t+c(t))(u^+)^{p-1} -c(t) (u^-)^{p-1} &&\text{in }~ \Om,\\
u &= 0 &&\text{on }~ \partial \Om.
\end{aligned}
\right\}
\end{equation*}
By Theorem 2.1 of \cite{Cuesta-Nodal}, we know that $u$ has exactly two nodal domains, $N^+ := \{x\in \Om : u(x)>0\}$ and $N^- := \{x\in \Om : u(x)<0\}.$ 
Since the restriction of $u$ to each of the nodal domains is an eigenfunction of $-\De_p$ with a constant sign, we easily get 
\begin{equation}\label{eqn10}
\la_1(N^+)=t+c(t)
\text{ and }
\la_1(N^-)=c(t).
\end{equation}
Since $u$ is radial and $\Om$ is radially symmetric, the nodal domains are also radially symmetric. 
Assume for definiteness that $u$ is negative near the outer boundary of $\Om.$ 
Thus there exists $R>0$ such that $N^+=\{x\in \Om: |x|<R \}$ and $N^-=\{x\in \Om:|x|>R\}.$ If $\Om$ is a ball, say $B_{R_1}(0),$ then $N^+= B_R(0)$ and $N^-=A_{R_1,R}(0,0).$ 
Now for $s\in (0,R_1-R),$ 
by using \eqref{eqn10} and Theorem \ref{MainTHM}  we obtain $\la_1(B_{R}(se_1))=t+c(t)$ and  $\la_1(A_{R_1,R}(0,se_1))< c(t)$. Further, using the continuity of $\la_1(\Omega)$ (see, for instance, Theorem 1 of \cite{Melian2001})  
we can find $\wide{R}\in (R,R_1)$ such that 
\begin{equation*}\label{ineq1}
 \la_1(B_{\wide{R}}(se_1))< t+ c(t)
\text{ and }
 \la_1(A_{R_1,\wide{R}}(0,se_1))< c(t).
\end{equation*}
If $\Om$ is an annulus, say $A_{R_1,R_0}(0,0),$ then we have $N^+= A_{R,R_0}(0,0)$ and $N^-=A_{R_1,R}(0,0).$ Now for $0<s< \min\{R_1-R,R-R_0\}$ by using \eqref{eqn10} and Theorem \ref{MainTHM}  we obtain 
\begin{equation*}\label{ineq2}
 \la_1( A_{R,R_0}(se_1,0))< t+ c(t)
\text{ and }
 \la_1(A_{R_1,R}(0,se_1))< c(t).
\end{equation*}
In either case, we have two disjoint domains $\Om_1$ and $\Om_2$ such that 
$$
\la_1(\Om_1)<t+c(t)
\text{ and }
\la_1(\Om_2)<c(t).
$$ 
Let $u_1$ and $u_2$ be corresponding eigenfunctions. Clearly $u_1$ and $u_2$ have disjoint supports and 
$$
\int\limits_{\Omega}|\nabla u_1|^p \dx <  (t+ c(t))\int\limits_{\Omega}|u_1|^p \dx
\text{ and }
\int\limits_{\Omega}|\nabla u_2|^p \dx <  c(t) \int\limits_{\Omega}|u_2|^p \dx.
$$
The above inequalities lead to a contradiction to the definition \eqref{eq:Fucik} of $c(t)$ by the same arguments as in the proof of Theorem 3.1  of  \cite{Cuesta-Fucik}. Thus $u$ must be nonradial. This completes the proof. 
\qed

\addcontentsline{toc}{section}{\refname}
\bibliographystyle{abbrv}
\bibliography{ref1}

\medskip
\medskip
\noindent
\textbf{T. V. Anoop}:\quad
	Department of Mathematics, Indian Institute of Technology Madras, Chennai 36, India.\\  
	\textit{Email}: \texttt{anoop@iitm.ac.in}
	
\medskip
\noindent
\textbf{Vladimir Bobkov}:\quad
	Department of Mathematics and NTIS, Faculty of Applied Sciences, University of West Bohemia,
	Univerzitn\'i 8, Plze\v{n} 306 14, Czech Republic.  \\
	Institute of Mathematics, Ufa Scientific Center, Russian Academy of Sciences, Chernyshevsky str. 112, Ufa 450008, Russia. \\
	\textit{Email}: \texttt{bobkov@kma.zcu.cz} 

\medskip
\noindent
\textbf{Sarath Sasi}:\quad 
	School of Mathematical Sciences, National Institute for Science Education and Research Bhubaneswar, Jatni 752050, India. \\
	\textit{Email}: \texttt{sarath@niser.ac.in}

\end{document}